\documentclass[12pt]{article} 

\usepackage[margin=2.5cm]{geometry}
\usepackage[english]{babel}
\usepackage{hyperref}
\usepackage{amsmath}
\usepackage{amsthm}
\usepackage{amssymb}
\usepackage{ucs}
\usepackage{enumerate}
\usepackage[utf8x]{inputenc}
\usepackage[T1]{fontenc}
 
\DeclareMathOperator{\supp}{\mathrm{supp}}

\newcommand{\Aut}{\mathrm{Aut}}

  \newcommand{\R}{\mathbb R}
  \newcommand{\N}{\mathbb N}
  \newcommand{\Q}{\mathbb Q}
  \newcommand{\Z}{\mathbb Z}
  
  \newcommand{\LL}{\mathrm L}

  \newcommand{\inv}{^{-1}}

  \renewcommand{\ker}{\mathrm{Ker}\,}

  \renewcommand{\leq}{\leqslant}
  \renewcommand{\geq}{\geqslant}
  \newcommand{\abs}[1]{\left\lvert #1\right\rvert}
  \newcommand{\norm}[1]{\left\lVert #1\right\rVert}

  \newcommand{\la}{\left\langle}
  \newcommand{\ra}{\right\rangle}
  
  \newcommand{\onto}{\twoheadrightarrow}

\newcommand{\BN}{{\mathbb{N}}}
\newcommand{\BR}{{\mathbb{R}}}
\newcommand{\BC}{{\mathbb{C}}}

\newcommand{\gb}{\beta}

\newcommand{\gC}{\Gamma}

\newcommand{\gO}{\Omega}

\newcommand{\ga}{\alpha}

\newcommand{\SL}{\text{SL}}

\def\Aut{\text{Aut}}
\def\Out{\text{Out}}

\newtheorem{thm}{Theorem}[section]
\newtheorem{cor}[thm]{Corollary}
\newtheorem{lem}[thm]{Lemma}
\newtheorem{prop}[thm]{Proposition}
\newtheorem*{clm}{Claim}

\theoremstyle{definition}

\newtheorem{qu}[thm]{Question}

\newtheorem*{ack}{Acknowledgements}

\newtheorem{defn}[thm]{Definition}
\newtheorem*{rmq}{Remark}
\newtheorem{exam}[thm]{Example}

\title{Infinitesimal topological generators and quasi non-archimedean topological groups}

\author{Tsachik Gelander\footnote{Research supported by the ISF, Moked grant
2095/15.}  and François Le Maître\footnote{Research supported by the Interuniversity Attraction Pole DYGEST and Projet ANR-14-CE25-0004 GAMME.}}
\date{}
\begin{document}

\maketitle

\begin{abstract}
We show that connected separable locally compact groups are infinitesimally finitely generated, meaning that there is an integer $n$ such that every neighborhood of the identity contains $n$ elements generating a dense subgroup.
We generalize a theorem of Schreier and Ulam by showing that any separable connected compact group is infinitesimally $2$-generated.

 Inspired by a result of Kechris, 
we introduce the notion of a quasi non-archimedean group. We observe that full groups are quasi non-archimedean, and that every continuous homomorphism from an infinitesimally finitely generated group into a quasi non-archimedean group is trivial. 
We prove that a locally compact group is quasi non-archimedean if and only if it is totally disconnected, and provide various examples which show that the picture is much richer for Polish groups. In particular, we get an example of a Polish group which is infinitesimally $1$-generated but totally disconnected, strengthening Stevens' negative answer to Problem 160 from the Scottish book. 

\end{abstract}

\section{Introduction}

One of the simplest invariants one can come up with for a topological group $G$ is its topological rank $t(G)$, that is, the minimum number of elements needed to generate a dense subgroup of $G$. For this invariant to have a chance to be finite, one needs to assume that $G$ is separable since every finitely generated group is countable.

The oldest result in this topic is probably due to Kronecker in 1884 \cite{zbMATH03002825}, and says that an $n$-tuple $(a_1,...,a_n)$ of real numbers projects down to a topological generator of $\mathbb T^n=\R^n/\Z^n$ if and only if  $(1,a_1,...,a_n)$ is $\Q$-linearly independent. Since such $n$-tuples always exist, the topological rank of the compact connected abelian group $\mathbb T^n$ is equal to one.

Using this result, it is then an amusing exercise to show that $t(\R^n)=n+1$. This means that as a topological group, $\R^n$ remembers its vector space dimension. In particular, we see that the topological rank sometimes contains useful information (another somehow similar instance of this phenomenon was recently discovered by the second-named author for full groups, see \cite{gentopergo}).

In order to deal with general connected Lie groups, it is useful to introduce the following definition.

\begin{defn}
The \textbf{infinitesimal rank} of a topological group $G$ is the minimum $n\in\N\cup\{+\infty\}$ such that every neighborhood of the identity in $G$ contains $n$ elements $g_1,...,g_n$ which generate a dense subgroup of $G$. We denote it by $t_I(G)$. 
\end{defn}

One can easily show that $t_I(\R^n)=t(\R^n)=n+1$. This is relevant for the study of the infinitesimal rank of real Lie groups because if we know that the Lie algebra of a connected Lie group $G$ is generated (as a Lie algebra) by $n$ elements, then using the fact fact that $t_I(\R)=2$ we can deduce that $t_I(G)\leq 2n$ (see the well-known Lemma \ref{prop:liealgtopogen}).

For various classes of connected Lie groups one can say more. In particular this is the situation for compact connected Lie groups, in which case a uniform result is useful in view of the important role of compact groups in the structure theory of locally compact groups.
 Auerbach showed that for every compact connected Lie group $G$, one has $t_I(G)\leq 2$ \cite{Auerbach1934}. Moreover, he could show that the set of pairs of topological generators of $G$ has full measure, which then led Schreier and Ulam to the following general result, for which we will include a short proof.

\begin{thm}[Schreier-Ulam, \cite{Schreier1935}]\label{thm:SUoriginal}
Let $G$ be a connected compact metrisable group. Then almost every pair of elements of $G$ generates a dense subgroup of $G$. In particular, $t_I(G)=t(G)=2$ for any non-abelian such $G$.
\end{thm}

 Note that there is a vast area between metrizable compact groups and separable ones; for instance ${(\mathbb T^1)}^\R$ is compact separable, but not metrizable. Conversely, being separable is a minimal assumption for a group to have finite topological rank.
Going back to the $n$-torus, Kronecker's result admits a far-reaching generalization due to Halmos and Samelson, which settles the abelian case.

\begin{thm}[Halmos-Samelson, {\cite[Corollary]{MR0006543}}] The topological rank of every connected separable compact abelian group is equal to one.
 \end{thm}

Our first result is a generalization of Schreier and Ulam's Theorem to the separable compact case. One of the difficulties is that we cannot use their Haar measure argument anymore.

\begin{thm}[see Theorem \ref{thm:SU}]\label{thm:SU1}
Let $G$ be a separable compact connected group. If $G$ is non-abelian, then $t_I(G)=2$, and if $G$ is abelian, then $t_I(G)=1$. 
\end{thm}

The fact that the topological rank $t(G)=2$ was proved by Hoffmann and
Morris \cite[Theorem. 4.13]{MR1082789}. Note that it follows that $t_I(G)=t(G)$ for separable compact connected group.

 It is natural to ask if the following stronger statement is true:

\begin{qu}
Let $G$ be a separable compact connected group. Is the set of pairs in $G$ which topologically generate $G$ necessarily of full measure in $G\times G$?
\end{qu}

Say that a topological group is \textbf{infinitesimally finitely generated} if it has finite infinitesimal rank. Using the previous result, and the solution to Hilbert's fifth problem, we establish the following result, which was noted by Schreier and Ulam in the abelian case (see the last paragraph of
\cite{Schreier1935}). 
\begin{thm}[see Theorem. \ref{thm:infinitesimally finitely generated2}]\label{thm:infinitesimally finitely generated}
Let $G$ be a separable  locally compact group. Then $G$ is infinitesimally finitely generated if and only if $G$ is connected.
\end{thm}

Let us now leave the realm of locally compact groups for a moment and discuss the class of \textbf{Polish groups}, i.e. separable topological groups whose topology admits a compatible metric. These groups abound in analysis, for instance the unitary group of a separable Hilbert space or the group of measure-preserving transformations of a standard probability space are Polish groups. Moreover, they form a robust class of groups, e.g. every countable product of Polish groups is Polish (see \cite{MR2455198} for other properties of this flavour). Recall that a locally compact group is Polish if and only if it is second-countable.  

It is not hard to show that Theorem \ref{thm:infinitesimally finitely generated} fails for Polish groups: for instance $\R^\N$ is connected but not topologically finitely generated, in particular it is not infinitesimally finitely generated. The question of the converse is more interesting, even for the following weaker property. 

\begin{defn}
A topological group $G$ is \textbf{infinitesimally generated} if every neighbourhood of the identity generates $G$.
\end{defn}

Clearly every connected group is infinitesimally generated, and every infinitesimally finitely generated group is infinitesimally generated. Moreover it follows from van Dantzig's theorem that every infinitesimally generated locally compact group is connected.

\begin{qu}[Mazur's Problem 160 \cite{MR666400}]\label{qu: inf gen is connected}Must an infinitesimally generated Polish group be connected? 
\end{qu}

In \cite{zbMATH03966486}, Stevens exhibited the first examples of infinitesimally generated Polish group which are totally disconnected. We show that her examples actually have infinitesimal rank $2$ and then provide the following stronger negative answer to Question \ref{qu: inf gen is connected}.

\begin{thm}[see Theorem. \ref{thm:infrank1 but td}]\label{thm: strong no to mazur problem}There exists a Polish group of infinitesimal rank $1$ which is totally disconnected. 
\end{thm}

Let us now introduce the quasi non-archimedean property which is a strong negation of being infinitesimally finitely generated. 

\begin{defn}
A topological group is \textbf{quasi non-archimedean} if for every neighborhood of the identity $U$ in $G$ and every $n\in\N$, there exists a neighborhood of the identity $V$ such that for every $g_1,...,g_n\in V$, the group generated by $g_1,...,g_n$ is contained in $U$.
\end{defn}
\begin{rmq}
If we switch the quantifiers and ask for a $V$ which works for every $n\in\mathbb N$, it is not hard to see that the definition then becomes that of a non-archimedean topological group (i.e. admitting a basis of neighborhoods of the identity made of open subgroups). 
\end{rmq}

Our inspiration for the above definition comes from the following result of Kechris: every continuous homomorphism from an infinitesimally finitely generated group into a full group is trivial (see the paragraph just before Section (E) of Chapter 4 in \cite{MR2583950}). We upgrade this by showing that every full group is quasi non-archimedean, and that  any continuous homomorphism from an infinitesimally finitely generated group into a quasi non-archimedean group is trivial (see Proposition \ref{prop:QNAtrivial}). For locally compact groups, we obtain the following characterisation. 

\begin{thm}[see Theorem. \ref{thm: cara lc qna or ltfg}] Let $G$ be a separable locally compact group. Then $G$ is quasi non-archimedean if and only if $G$ is totally disconnected.
\end{thm}

Note that full groups are connected and at the same time quasi non-archimedean Polish groups. Moreover, we show that every quasi non-archimedean Polish groups embeds into a connected quasi non-archimedean Polish group (see Proposition \ref{prop:embed in QNA}.).

We also provide examples of totally disconnected Polish groups which are quasi non-archimedean, but not non-archimedean (see Corollary. \ref{Corollary:exists td qna not na}). On the other hand Theorem  \ref{thm: strong no to mazur problem} ensures us that there are totally disconnected Polish groups which are not quasi non-archimedean.\\

The paper is organised as follows. In Section \ref{sec:basics}, we prove some basic results on topological generators. In Section \ref{sec:connectedcompact}, we show that separable connected compact groups are infinitesimally $2$-generated. Section \ref{sec:locgeninlc} is devoted to the proof that every connected separable locally compact group is infinitesimally finitely generated. In Section \ref{sec:QNA} we introduce quasi non-archimedean groups and study their basic properties. We also give numerous examples, and show that a separable locally compact group is totally disconnected if and only if it is quasi non-archimedean. Finally, a Polish group into which no non-discrete locally compact group can embed is built in Section \ref{sec:remarks}, where we also ask three questions raised by this work.

\begin{ack} We warmly thank Pierre-Emmanuel Caprace for coming up with the terminology ‘‘infinitesimally generated'', as well as for useful conversations around this topic. We also thank Yves de Cornulier for his helpful remarks on a first version of the paper. 
\end{ack}

\section{Basic results about topological generators}\label{sec:basics}
We collect some results which will serve us in the preceding sections.

\begin{prop}\label{prop:K-profinite}
Let $G$ be a connected locally compact group. Suppose that $K$ is a profinite normal subgroup of $G$ such that $G/K$ is infinitesimally finitely generated. Then $t_I(G)=t_I(G/K)$. Moreover, if $\bar{g}_1,\ldots,\bar{g}_k$ topologically generate the group $G/K$, and $g_1,\ldots,g_k$ are arbitrary respective lifts in $G$, then $g_1,\ldots,g_k$ topologically generate $G$.
\end{prop}

The proof of Proposition \ref{prop:K-profinite} will rely on the following two lemmas:

\begin{lem}\label{lem:finite-cover}
Let $H$ be a connected locally compact group and $f:H\onto L$ a finite covering map. Let $\{ l_1,\ldots l_k\}$ be a topological generating set for $L$ and pick $h_i\in f^{-1}(l_i),~i=1,\ldots,k$ arbitrarily. Then $h_1,\ldots,h_k$ topologically generate $H$.
\end{lem}

\begin{proof}
Set $F=\overline{\langle h_1,\ldots,h_k\rangle}$. Note that $f$, being a finite cover, is a closed map, and hence $f(F)$ is closed in $L$. Since $l_1,\ldots,l_k\in f(F)$ we have $f(F)=L$. Thus $\ker{f}\cdot F=H$. Hence by the Baire category theorem $F$ has a non-empty interior. Since $H$ is connected, this implies that $F=H$. 
\end{proof}

\begin{lem}\label{lem:pro-central}
Let $G$ be a (connected) locally compact group admitting a pro-finite normal subgroup $K\lhd G$ such that $L=G/K$ is a Lie group. Then $G$ is an inverse limit $G=\varprojlim L_\ga$
of finite (central) extensions $L_\ga$ of $L$.
\end{lem}

Although the Lemma holds without the assumption that $G$ is connected, since it significantly simplify the proof while being sufficient for our needs, we will prove it only under the connectedness assumption.

\begin{proof}
Given $k\in K$, the image of the orbit map $G\to K,~g\mapsto gkg^{-1}$ is at the same time connected, since $G$ is connected and the map is continuous, and totally disconnected as the image lies in $K$. It follows that it is constant and $k$ is central. Since $k$ is arbitrary we deduce that $K$ is central in $G$. In particular every subgroup of $K$ is normal in $G$. 

Let $K_\ga$ be a net of open subgroups in $K$ with trivial intersection. Then $K=\varprojlim K/K_\ga$ and $G=\varprojlim G/K_\ga$. Set $L_\ga=G/K_\ga$ and note that as $K_\ga$ is open in $K$, it is of finite index there. Thus the map $L_\ga\to L$ is a finite covering.
\end{proof}

\begin{proof}[Proof of Proposition \ref{prop:K-profinite}]
Let $G$ and $K$ be as in Proposition \ref{prop:K-profinite}. By Lemma \ref{lem:pro-central}, $G=\varprojlim L_\ga$ is an inverse limit of finite covers $L_\ga$ of $L=G/K$. Let $\bar{g}_1,\ldots,\bar{g}_k$ be topological generators of $L$, let $g_1,\ldots,g_k$ be arbitrary lifts in $G$ and denote by $g_i^\ga$ the projection of $g_i$ in $L_\ga$, for every $i,\ga$. By Lemma \ref{lem:finite-cover}, $\langle g_i^\ga:i=1,\ldots,k\rangle$ is dense in $L_\ga$. That is, the group $\langle g_i:i=1,\ldots,k\rangle$ projects densely to all $L_\ga$. This implies that it is dense in $G$. 
\end{proof}

\section{Connected compact separable groups are infinitesimally $2$-generated}\label{sec:connectedcompact}

Our aim in this section is to prove the following result. Recall that $t(G)$ is the minimal number of topological generators of $G$, while $t_I(G)$ is the minimal $n\in\N$ such that every neighborhood of the identity contains $n$ elements which topologically generate $G$.

\begin{thm}\label{thm:SU}
Let $G$ be a separable compact connected group. Then $t_I(G)=t(G)=2$ if $G$ is nonabelian and $t_I(G)=t(G)=1$ if $G$ is abelian. 
\end{thm}

\subsection{The metrizable case}

The case where $G$ is metrizable was proved by Schreier and Ulam \cite{Schreier1935}. Recall that a compact group is metrizable if and only if it is first countable. In that case one can show that 
almost every pair of elements in $G$ (or a single element if $G$ is abelian) topologically generates $G$.
Let us give a short argument for Theorem \ref{thm:SUoriginal} in that case. 
First recall that by the Peter--Weyl theorem, $G$ is an inverse limit of compact Lie groups and, being first countable, the limit is over a countable net. Since a countable intersection of full measured sets is of full measure, it is enough to prove the analog statement for compact connected Lie groups.

Note also that, in complete generality, a subgroup $H\le G$ is dense if and only if
\begin{itemize}
\item $H\cap G'$ is dense in $G'$, where $G'$ is the commutator subgroup in $G$, and
\item $HG'$ is dense in $G$.
\end{itemize}

A connected compact Lie group $G$ is reductive, hence an almost direct product of its commutator $G'$ with its centre $Z$. Moreover $G'$ is connected and a finite cover of $G/Z$ which is a semisimple group of adjoint type. In view of Lemma \ref{lem:finite-cover}, we deduce:

\begin{lem}\label{lem:1.9}
Let $G$ be a connected compact Lie group and $H\le G$ a subgroup. Then $H$ is dense in $G$ if and only if both $HG'$ and $HZ(G)$ are dense in $G$.
\end{lem}

Therefore, for a pair $(x,y)\in G^2$ to generate a dense subgroup, it is sufficient if
\begin{itemize}
\item the projections of $x,y$ to $G/Z$ generate a dense subgroup in $G/Z$, and
\item the projection of $x$ to $G/G'$ generates a dense subgroup in $G/G'$.
\end{itemize}
It is easy to check that both conditions are satisfied with Haar probability $1$ (cf. \cite[Lem. 1.4 and Lem. 1.10]{zbMATH05508698}).


\subsection{Proof of Theorem \ref{thm:SU} in the general case}

Let us first deal with abelian groups. By the Halmos--Samelson theorem, whenever $G$ is connected compact separable abelian, one has $t(G)=1$. From their result, we deduce the following consequence. 

\begin{lem}\label{lem:tlabelian}Let $G$ be a compact connected separable abelian group. Then $t_I(G)=1$.
\end{lem}
\begin{proof}
By the Halmos--Samelson theorem, we may and do pick $g\in G$ such that $\la g\ra$ is dense in $G$. Let $U$ be a neighborhood of the identity in $G$, and fix $n\in\Z\setminus\{0\}$ such that $g^n\in U$. Then since $\la g^n\ra$ has finite index in $\la g\ra$, its closure $\overline{\la g^n\ra}$ has finite index in $\overline{\la g\ra}=G$. Since $G$ is connected, we must have $\overline{\la g^n\ra}=G$.
\end{proof}

Now suppose that $G$ is any connected compact group.
By the Peter--Weil theorem $G=\varprojlim G_\ga$ is an inverse limit of compact connected Lie groups $G_\ga$.

Observe that a surjective map $f:G_1\onto G_2$ between groups always satisfies 
$$
 f(G_1')=G_2'~\text{and}~ f(Z(G_1))\subset Z(G_2),
$$ 
while for reductive Lie groups we also have $f(Z(G_1))=Z(G_2)$.
Since every connected compact Lie group is reductive, hence the product of its centre and its commutator, we deduce that the same hold for general compact connected groups, i.e. 
$$
 Z(G)=\varprojlim Z(G_\ga),~ G'=\varprojlim G_\ga'~\text{and}~G=Z(G)G'.
$$
 
Moreover, since the $G_\ga'$ are semisimple, and in particular perfect, $G'$ is also perfect, i.e. $G'=G''$. It follows that if $\gC\le G'$ is dense, then $\gC'$ is dense in $G'$. Hence we have:

\begin{clm}
Suppose that $a,b\in G'$ topologically generate $G'$, and $h\in Z(G)$ is an element whose image mod $G'$ topologically generates $G/G'$. Then $ah$ and $b$ topologically generate $G$.
\end{clm}

%

Suppose from now on that $G$ is separable. Then every quotient of $G$ is also separable. Now $G/G'$ is connected and abelian, so  we have  $t_I(G/G')=1$ by Lemma \ref{lem:tlabelian}.

Since $Z(G)$ surjects onto $G/G'$, every identity neighbourhood in $G$ contains a central element $h$ whose image in $G/G'$ generate a dense subgroup. Thus
we are left to show that $t_I(G')=2$. The centre of $G'$ is totally disconnected since it can be written as $Z(G')=\varprojlim Z(G_\ga')$, and every $G_\ga'$ has finite center. In view of Proposition \ref{prop:K-profinite} we may thus suppose that $G'$ is center-free.
In order to simplify notations, let us suppose below that $G$ itself is center-free. Note that:

\begin{lem}
A center-free connected compact group is a direct product of simple Lie groups.\qed
\end{lem}

Thus, $G$ is of the form $G=\prod_{\ga\in I} S_\ga$ with $S_\ga$ being connected adjoint simple Lie group. 

\begin{lem}
The group $G=\prod_{\ga\in I} S_\ga$ is separable (if and) only if $\text{Card}(I)\le 2^{\aleph_0}$.
\end{lem}

Let us explain the `only if' side. The other direction will follow once we show that $\text{Card}(I)\le 2^{\aleph_0}$ implies that $G$ has a two generated dense subgroup.
Suppose by way of contradiction that $\text{Card}(I)> 2^{\aleph_0}$. Since there are only countably many isomorphism types of (compact adjoint) simple Lie groups, we deduce that there is some compact simple Lie group $S$ and a cardinality $\kappa>2^{\aleph_0}$ such that $G$ admits a factor isomorphic to $S^\kappa$. However by the cardinals version of the pigeon hole principal, if $D\subset S^\kappa$ is a countable subset, there must be two factors $S_1,S_2$ of $S^\kappa$ such that the projection of $D$ to $S_1\times S_2$ lies in the diagonal. In particular, $D$ cannot be dense, confirming the desired contradiction. \qed

Thus, we may suppose below that $\text{Card}(I)\le 2^{\aleph_0}$. 

\begin{defn}
Let $S_1,S_2$ be two groups, $F$ a set and $f_i:F\to S_i,~i=1,2$ two maps. We shall say that the maps $f_1$ and $f_2$ are isomorphically related if there is an isomorphism $\phi:S_1\to S_2$ such that the following diagram is commutative
\[
\begin{array}[c]{ccccc}
F&\xrightarrow{f_1}&S_1\\
&\searrow\scriptstyle{f_2}&\downarrow\scriptstyle{\phi}&\\
&&S_2
\end{array}
\]
\end{defn}

\begin{lem}\label{lem:eso-rel}
Let $\prod_{i=1}^k S_i~(k\ge 2)$ be a product of simple groups and let $R< \prod_{i=1}^k S_i$ be a proper subgroup that projects onto every $S_i$. Then there are two factors $S_i,S_j,~i\ne j$ such that the restrictions of the quotient maps $R\to S_i$ and $R\to S_j$ are isomorphically related.
\end{lem}

\begin{proof}
For a subset $J\subset \{1,\ldots,k\}$ let us denote $S_J=\prod_{i\in J}S_i$ and $R_J=\text{Proj}_{S_J}(R)$.
Let $J\subset \{1,\ldots,k\}$ be a minimal subset such that $R_J$ is a proper subgroup of $S_J$. By our assumption $J$ exists and satisfies $|J|>1$. We claim that $|J|=2$.  
To see this, we may reorder the indices so that $J=\{1,\ldots,j\}$ and suppose by way of contradiction that $j\ge 3$. This, together with the minimality of $J$, implies that for every $g\in S_1$ and every $1<i\le j$ there is an element in $R_J$ whose first coordinate is $g$ and whose $i$'th coordinate is $1$. However, multiplying commutators of elements as above, forcing that at each coordinate $1<i\le j$ at least one of the elements we use is trivial, and using the fact that $S_1$ is perfect, we deduce that $S_1\le R_J$. In the same way we get that $S_i\le R_J$ for all $i\in J$ contradicting the assumption that $R_J$ is proper in $S_J$.  
Thus $J=\{1,2\}$. Moreover, as $S_i$ is simple and normal, and $R_J$ projects onto $S_i$, while $R_J$ is proper in $S_J$, it follows that $R_J\cap S_i$ is trivial, for $i=1,2$.
Therefore, the restriction of the projection from $R_J$ to $S_i$ is an isomorphism, for $i=1,2$, hence the maps $R\to S_1$ and $R\to S_2$ are isomorphically related.
\end{proof}

Assembling together isomorphic factors of $G$, we may decompose $G$ as a direct product $G=\prod_{n\in\BN} G_n$ where different $n$'s correspond to non-isomorphic simple Lie factors $S_n$, and each $G_n$ is of the form $G_n=S_n^{I_n}$ with $|I_n|\le 2^{\aleph_0}$. For $x\in G$ we shall denote by $x_n^\ga,~\ga\in I_n$ its corresponding coordinates.

Two elements $x,y\in G$ generate a dense subgroup if and only if for any finite set of indices $F=\{ (n_i,\ga_i)\}$ the projections of $x$ and $y$ generate a dense subgroup in $\prod_F S_{n_i}^{\ga_i}$. In view of Lemma \ref{lem:eso-rel} we have:

\begin{cor}
Two elements $x,y\in G$ generate a dense subgroup in $G$ if and only if
\begin{itemize}
\item $\langle x_n^\ga,y_n^\ga\rangle$ is dense in $S_n^\ga$ for every $n\in \BN$ and every $\ga\in I_n$, and
\item For every $n,m\in\BN$ and $\ga\in I_n,\gb\in I_m$ the (projection) maps $\{x,y\}\to S_n^\ga$, $\{x,y\}\to S_m^\gb$ are not isomorphically related.
\end{itemize}
\end{cor}
 
Thus, in order to prove the theorem, we should explain how to pick $(x_n^\ga,y_n^\ga)$ arbitrarily close to the identity in $S_n^\ga$, for every $n\in\BN,\ga\in I_n$, which generate a dense subgroup of $S_n$, such that for every pair $(n,\ga)\ne(m,\gb)$ there is no isomorphism $S_n\to S_m$ taking $(x_n^\ga,y_n^\ga)$ to $(x_m^\gb,y_m^\gb)$. Since for $n\ne m$ there is no isomorphism between $S_n$ and $S_m$ we should only consider the case $n=m$.

Let $S=S_n$. Recall that there is an identity neighbourhood (a Zassenhaus neighbourhood) $\gO\subset S$ on which the logarithm is well defined, and for $x,y\in \gO$ the group $\langle x,y\rangle$ is dense in $S$ if $\log(x)$ and $\log (y)$ generate the Lie algebra $\text{Lie}(S)$, (see \cite{zbMATH03069311}). Thus we may pick some open set $U_1\times U_2\subset \gO^2$ such that for all $(x,y)\in U_1\times U_2$ the group $\langle x,y\rangle$ is dense in $S$ (cf. \cite{zbMATH01760707}). Furthermore, we may pick $U_1,U_2$ arbitrarily close to the identity. Since $\Out(S)$ is finite and $S$ is compact, we may also suppose, by taking $U_1$ and $U_2$ sufficiently small, that for $(x,y)\in U_1\times U_2$, if $f\in \Aut(S)$ satisfies $(f(x),f(y))\in U_1\times U_2$, then $f$ is inner. Finally, since the orbit of each $(x,y)\in U_1\times U_2$ under the action of $S$ on $S\times S$ by conjugation, is $\dim(S)$ dimensional, while $\dim (U_1\times U_2)=2\dim(S)$, we can pick a section $\mathcal{S}$  transversal to the orbits foliation, inside $U_1\times U_2$. Clearly the cardinality of that section is $2^{\aleph_0}$, hence we may imbed $I_n$ inside this section. This imbedding yields a choice of $(x^\ga,y^\ga)\in \mathcal{S}\subset U_1\times U_2$ for every $\ga\in I_n$, and we have that each pair $(x^\ga,y^\ga)$ generates a dense subgroup in $S$, and no two pairs are isomorphically related. This completes the proof of Theorem \ref{thm:SU}.
\qed

\section{Local generators and locally compact connected groups}\label{sec:locgeninlc}

Recall that a topological group $G$ is \textbf{infinitesimally finitely generated} if $t_I(G)<+\infty$, that is, if there exists $n\in\N$ such that every neighborhood of the identity contains $n$ topological generators for $G$. 

Similarly, we say that $G$ is \textbf{topologically finitely generated} if $t(G)<\infty$, that is, if $G$ admits a dense finitely generated subgroup.

Last but not least, a topological group is \textbf{infinitesimally generated} if it is generated by every neighborhood of the identity.

Note that if a group is infinitesimally finitely generated, then it is also topologically finitely generated and infinitesimally generated. Our main goal in this section is to prove the following result.

 \begin{thm}\label{thm:infinitesimally finitely generated2}
Let $G$ be a separable connected locally compact group. Then $G$ is infinitesimally finitely generated. 
\end{thm}

As already noted, the separability condition is necessary for a group to be topologically finitely generated. The connectedness assumption is also necessary (for local generation) since otherwise $G/G^\circ$, the group of connected components, is non-trivial and by the van Dantzig's theorem admits a base of identity neighbourhood consisting of open compact subgroups. In particular, if $O\le G/G^\circ$ is a proper open subgroup, its pre-image in $G$ cannot contain a topological generating set.
\medskip

Let us start by dealing with the nicest connected locally compact groups: Lie groups. For these, one can use the Lie algebra to produce topological generators. The following lemma is well known, but we include a proof for the reader's convenience.

\begin{lem}[Folklore]\label{prop:liealgtopogen}
Let $G$ be a connected Lie group, and let $\mathfrak g$ be its Lie algebra. Suppose that $\mathfrak g$ is generated as a Lie algebra by $X_1,...,X_n$. Then every neighborhood of the identity in $G$ contains $2n$ elements $g_1,...,g_{2n}$ which generate a dense subgroup in $G$. 
In particular, $G$ is infinitesimally finitely generated.
\end{lem}
\begin{proof}
Let $V$ be a neighborhood of the identity in $G$.
Let $U$ be a small enough neighborhood of 0 in $\mathfrak g$ such that $\exp: U\to G$ is a homeomorphism onto its image, and $\exp(U)\subseteq V$. Fix $2n$ elements $Y_1,...,Y_{2n}$ of $U$ such that for every $i\in\{1,...,n\}$, $\{Y_{2i},Y_{2i+1}\}$ generates a dense subgroup of $\R X_i$. 

For all $i\in\{1,...,2n\}$, let $g_i=\exp(Y_i)$, we will show that these elements topologically generate $G$. Let $H$ be the closed subgroup generated by the set $\{g_i\}_{i=1}^{2n}$. Note that for all $i\in\{1,...,n\}$, the group $H$ contains $\exp(\R X_i)$ since the restriction of the exponential map to $\R X_i$ is a continuous group homomorphism and $\R X_i$ is topologically generated by $Y_{2i}$ and $Y_{2i+1}$ which are mapped to $g_{2i}\in H$ and $g_{2i+1}\in H$.

Furthermore, $H$ is a Lie group by Cartan's theorem, and since $H$ contains every $\exp(\R X_i)$ the Lie algebra $\mathfrak h$ of $H$ contains every $X_i$, so $\mathfrak h=\mathfrak g$.  Because $G$ is connected, we get that $G=H$.
\end{proof}
We deduce from the above lemma that for any connected Lie group, $t_I(G)\leq 2\dim(G)$.
Better bounds on $t_I(G)$ can be deduced from the analysis in \cite{zbMATH01903603,MR2255501,zbMATH05508698}.

Let now $G$ be a general connected locally compact group. Recall the celebrated Gleason-Yamabe theorem (cf.  \cite[Page 137]{zbMATH03353499}):

\begin{thm}(Gleason--Yamabe)\label{Gleeson--Yamabe}
Let $G$ be a connected locally compact group. Then there is a compact normal subgroup $K\lhd G$ such that $G/K$ is a Lie group.
\end{thm}

Since $G$ is connected, $G/K$ is a connected Lie group and hence $t_I(G/K)$ is finite. However, $K$ may not be connected.

\begin{exam}
$(1)$ (The solenoid)
For every $n$, let $T_n$ be a copy of the circle group $\{z\in\BC:|z|=1\}$, and whenever $m$ divides $n$ let $f_{n,m}:T_n\to T_m$ be the $n/m$ sheeted cover $f_{n,m}(z)=z^{n/m}$. Let $T=\varprojlim T_n$ be the inverse limit group. Then $T$ is connected, abelian and locally compact, but admits no connected co-Lie subgroups.

$(2)$ Similarly, as $\SL_2(\BR)$ is homotopic to a circle, we can define, for every $n\in\BN$, $G_n$ as the $n$ sheeted cover of $\SL_2(\BR)$.
Then whenever $m$ divides $n$ there is a canonical covering morphisms $\psi_{n,m}:G_n\to G_m$, and we may let $G$ be the inverse limit $G=\varprojlim G_n$. Then $G$ is a connected locally compact group which admits no nontrivial connected compact normal subgroups. 
\end{exam}

In order to prove Theorem \ref{thm:infinitesimally finitely generated2}, we need one last elementary lemma.

\begin{lem}\label{lemG/N-infinitesimally finitely generated}
Let $G$ be a topological group and $N$ a normal subgroup. Then $t_I(G)\le t_I(N)+t_I(G/N)$. In particular if $G/N$ and $N$ are infinitesimally finitely generated then so is $G$. \qed
\end{lem}

\begin{proof}[Proof of Theorem \ref{thm:infinitesimally finitely generated2}]
Let $G$ be a connected separable locally compact group. Let $K\lhd G$ be a compact normal subgroup such that $G/K$ is a Lie group (see Theorem \ref{Gleeson--Yamabe}), and let $K^\circ$ be its identity connected component. Then $K^\circ$ is characteristic in $K$ and therefore normal in $G$. Being a closed subgroup of a locally compact separable group, $K^\circ$ is separable \cite{MR0447455}. Let $H=G/K^\circ$ and $K^t=K/K^\circ$. By the isomorphism theorem, $G/K\cong H/K^t$. Note that $K^t$ is a pro-finite group, hence by Proposition \ref{prop:K-profinite} $H$ is infinitesimally finitely generated. By Theorem \ref{thm:SU}, $K^\circ$ is infinitesimally finitely generated, hence, by Lemma \ref{lemG/N-infinitesimally finitely generated}, $G$ is infinitesimally finitely generated.
\end{proof}

\section{Quasi non-archimedean groups}\label{sec:QNA}

A topological group is \textbf{non-archimedean} if it has a basis of neighborhoods of the identity made of open subgroups. Equivalently,  every neighborhood of the identity $V$ contains a smaller neighborhood of the identity $U$ such that the group generated by $U$ is contained in $V$, which is the same as requiring that for every $n\in\N$ and $g_1,...,g_n\in U$, the group generated by $g_1,...,g_n$ is contained in $V$. The definition that follows is obtained by switching two quantifiers in the above condition.

\begin{defn}Say a topological group $G$ is \textbf{quasi non-archimedean} if for all $n\in\N$ and all neighborhood of the identity $V\subseteq G$, there exists a neighborhood of the identity $U\subseteq V$ such that for all $g_1,...,g_n\in U$, the group generated by $g_1,...,g_n$ is contained in $V$.
\end{defn}

Clearly every non-archimedean group is also quasi non-archimedean. Let us give right away the motivating example for this definition. We fix a standard probability space $(X,\mu)$, that is, a probability space which is isomorphic to the interval $[0,1]$ with its Borel $\sigma$-algebra and the Lebesgue-measure. 

A Borel bijection $T$ of $X$ is called a \textbf{non-singular automorphism} if for all measurable $A\subseteq X$, one has $\mu(A)=0$ if and only if $\mu(T\inv(A))=0$. The group of all these automorphisms is denoted by $\Aut^*(X,\mu)$, two such automorphisms being identified if they coincide on a full measure set. We then define the \textbf{uniform metric} $d_u$ on $\Aut^*(X,\mu)$ by: for all $T,U\in\Aut^*(X,\mu)$, 
$$d_u(T,U)=\mu(\{x\in X: T(x)\neq U(x)\}).$$
This is a complete metric, though far from being separable (e.g. the group $\mathbb S^1$ acts freely on itself, yielding an uncountable discrete subgroup of $(\Aut^*(X,\mu),d_u)$). But among closed subgroups of $(\Aut^*(X,\mu),d_u)$, full groups are separable.
Full groups are invariants of orbit equivalence attached to nonsingular actions of countable groups on $(X,\mu)$: given a non-singular action of a countable group $\Gamma$ on $(X,\mu)$, its full group $[\mathcal R_\Gamma]$ is the group of all $T\in\Aut^*(X,\mu)$ such that for every $x\in X$, $T(x)\in\Gamma\cdot x$.

Since every subgroup of a quasi non-archimedean group is quasi non-archimedean for the induced topology, the following result implies that full groups  are quasi non-archimedean. 

\begin{thm}[Kechris] \label{thm:AutQNA}$\Aut^*(X,\mu)$ is quasi non-archimedean for the uniform metric.
\end{thm}
\begin{proof}
Define the \textbf{support} of $T\in\Aut^*(X,\mu)$ to be the set of all $x\in X$ such that $T(x)\neq x$. Note that $d_u(\mathrm{id}_X,T)$ is precisely the measure of the support of $T$. 

Let $\epsilon>0$ and $n\in\N$, and consider the open ball $U:=B_{d_u}(\mathrm{id}_X,\epsilon)$. Suppose that $g_1,...,g_n$ belong to $V:=B_{d_u}(\mathrm{id_X},\epsilon/n)$, and let $A$ be the reunion of their supports. By assumption, $A$ has measure less than $\epsilon$. Then the group generated by $g_1,...,g_n$ is contained in the group of elements supported in $A$, which is itself a subset of $U=B_{d_u}(\mathrm{id}_X,\epsilon)$.
\end{proof}

The following proposition shows that the class of quasi non-archimedean groups satisfies basically the same closure properties as the class of non-archimedean groups. 

\begin{prop}\label{prop:QNAstable}The class of quasi non-archimedean groups is closed under taking subgroups (with the induced topology), products and quotients. \qed
\end{prop}

The next proposition highlights the main difference between non-archimedean and quasi non-archimedean groups. From now on, we will restrict ourselves to the narrower but well-behaved class of Polish groups, that is, separable groups whose topology admits a compatible complete metric, e.g. full groups for the uniform topology, the  group $\Aut^*(X,\mu)$ endowed with the weak topology, or the unitary group of a separable Hilbert space endowed with the strong operator topology. Let us point out that a locally compact group is Polish if and only if it is second-countable (see \cite[Thm. 5.3]{MR1321597}).

Recall that if $G$ is a Polish group and $(X,\mu)$ is a standard (non-atomic) probability space, then the group $\LL^0(X,\mu,G)$ of measurable maps from $X$ to $G$ is a Polish group for the topology of convergence in measure, two such maps being identified if they coincide on a full measure set. A basis of neighborhoods of the identity for this topology is given by the sets
$$\tilde U_\epsilon=\{f\in\LL^0(X,\mu,G): \mu(\{x\in X: f(x)\not\in U\})<\epsilon\},$$
where $U$ is an open neighborhood of the identity in $G$ and $\epsilon>0$. The Polish group $\LL^0(X,\mu,G)$ enjoys the two following nice properties (see e.g. \cite[Chap. 19]{MR2583950}) .
\begin{itemize}
\item $G$ embeds into $\LL^0(X,\mu,G)$ via constant maps.
\item $\LL^0(X,\mu,G)$ is connected, in fact contractible.
\end{itemize}

\begin{prop}\label{prop:embed in QNA}Let $G$ be a quasi non-archimedean Polish group. Then $\LL^0(X,\mu,G)$ is quasi non-archimedean. In particular any quasi non-archimedean Polish group embeds in a connected quasi non-archimedean Polish group. 
\end{prop}
\begin{proof}
Let $\tilde U_\epsilon=\{f: \mu(\{x\in X: f(x)\not\in U\})<\epsilon\}$ be a basic neighborhood of the identity in $\LL^0(X,\mu,G)$. Let $n\in\N$ and $V$ be a corresponding neighborhood of the identity witnessing that $G$ is quasi non-archimedean. Consider the following open neighborhood of the identity in $\LL^0(X,\mu,G)$: 
$$\tilde V_{\epsilon/n}=\{f: \mu(\{x\in X: f(x)\not\in V\})<\epsilon/n\}.$$
Then if we let $f_1,...,f_n\in \tilde V_{\epsilon,n}$, the reunion of the sets $\{x\in X: f_i(x)\not\in V\}$ has measure less than $\epsilon$. By the definition of $V$ and $U$ the group generated by $f_1,... , f_n$ is a subset of $\tilde U_\epsilon$.
\end{proof}

\begin{rmq}Since non-archimedean groups are totally disconnected, the above proposition implies there are a lot more quasi non-archimedean groups than the non-archimedean ones.
\end{rmq}

The following proposition is inspired by Section (D) of Chapter 4 in \cite{MR2583950}, where it is shown that any continuous homomorphism from an infinitesimally finitely generated group into a full group is trivial.

\begin{prop}\label{prop:QNAtrivial}Any continuous homomorphism from an infinitesimally finitely generated group into a quasi non-archimedean group is trivial.
\end{prop}
\begin{proof}
Let $\varphi: G\to H$ be such a morphism, let $V$ be any neighborhood of the identity in $H$, and let $n=t_I(G)$. Then there is a neighborhood of the identity $U$ in $H$ such that any subgroup of $H$ generated by $n$ elements of $U$ is contained in $V$. Since $\varphi$ is continuous, $\varphi\inv(U)$ is a neighborhood of the identity in $G$, and because $t_I(G)=n$ we may find $g_1,...,g_n\in\varphi\inv(U)$ which generate a dense subgroup in $G$. Then the closure of $\varphi(G)$ coincides with the closure of the group generated by $\varphi(g_1),...,\varphi(g_n)\in U$, which by assumption is contained in $\overline{V}$. So $\varphi(G)$ is contained in the closure of any neighborhood of the identity in $H$, and since $H$ is Hausdorff this means that $\varphi$ is trivial.
\end{proof}

\begin{cor}\label{cor: qna and ltfg are disjoint} 
The only topological group which is both infinitesimally finitely generated and quasi non-archimedean is the trivial group.\qed
\end{cor}

\begin{cor}Every continuous homomorphism from a connected separable locally compact group into $(\Aut^*(X,\mu),d_u)$ is trivial. 
\end{cor}
\begin{proof}Since every connected separable locally compact group is infinitesimally finitely generated by Theorem \ref{thm:infinitesimally finitely generated2} and $\Aut^*(X,\mu)$ is quasi non-archimedean by Theorem \ref{thm:AutQNA}, the previous proposition readily applies. 
\end{proof}

As a consequence of the previous proposition and Theorem \ref{thm:infinitesimally finitely generated}, we have the following interesting characterizations of connectedness and total disconnectedness for locally compact separable groups.

\begin{thm}\label{thm: cara lc qna or ltfg}Let $G$ be a locally compact separable group. Then the following hold:
\begin{enumerate}[(1)]
\item $G$ is connected if and only if $G$ is infinitesimally finitely generated.
\item $G$ is totally disconnected if and only if $G$ is quasi non-archimedean.
\end{enumerate}
\end{thm}
\begin{proof}
If $G$ is not connected but infinitesimally finitely generated, then $G/G^0$ must also be infinitesimally finitely generated, which is impossible by van Dantzig's theorem. The converse is provided by Theorem \ref{thm:infinitesimally finitely generated}.

If $G$ is totally disconnected, then $G$ is non-archimedean by van Dantzig's theorem. But this implies that $G$ is quasi non-archimedean. For the converse, suppose $G$ is quasi non-archimedean. Then $G^0$ also is, but then by (1) and Proposition \ref{prop:QNAtrivial} it must be trivial.
\end{proof}

\begin{rmq}
As was pointed out by Caprace and Cornulier, one can prove $(2)$ more directly. Indeed, if $G^0$ is non-trivial then it admits a non-trivial one-parameter subgroup which is in particular infinitesimally $2$-generated, contradicting Proposition \ref{prop:QNAtrivial}. This actually gives a proof that a locally compact group is quasi non-archimedean if and only if it is totally disconnected, regardless of its separability.
\end{rmq}

Let us now give an example of a totally disconnected Polish group which is quasi non-archimedean, but not non-archimedean. This class of examples was introduced by Tsankov \cite[Sec. 5]{MR2213623}, using work of Solecki \cite{MR1708146}. We denote by $\mathfrak S_\infty$ the group of all permutations of the integers, equipped with its Polish topology of pointwise convergence. Recall that every non-archimedean Polish group arises as a closed subgroup of $\mathfrak S_\infty$ (see \cite[Thm. 1.5.1]{MR1425877}). Here, the groups that we will consider are subgroups of $\mathfrak S_\infty$, but equipped with a Polish topology which refines the topology of pointwise convergence.

\begin{defn}
A \textbf{lower semi-continuous submeasure} on $\N$ is a function $\lambda: \mathcal P(\N)\to [0,+\infty]$ such that the following hold:
\begin{itemize}
\item $\lambda(\emptyset)=0$;
\item for all $n\in\N$, we have $0<\lambda(\{n\})<+\infty$;
\item for all $A\subseteq B\subseteq \N$, we have $\lambda(A)\leq \lambda(B)$;
\item (\textit{subadditivity}) for all $A,B\subseteq \N$, we have $\lambda(A\cup B)\leq \lambda(A)+\lambda(B)$;
\item (\textit{lower semi-continuity}) for every increasing sequence $(A_k)_{k\in\N}$ of subsets of $\N$, we have $\lambda(\bigcup_{k\in\N} A_k)=\lim_{k\in\N}\lambda(A_k)$.
\end{itemize}
\end{defn}

We associate to every lower semi-continuous submeasure $\lambda$ on $\N$ a subgroup of $\mathfrak S_\infty$, denoted by $\mathfrak S_\lambda$, defined by
$$\mathfrak S_\lambda=\{\sigma\in\mathfrak S_\infty: \lambda(\supp \sigma\setminus\{0,...,n\})\to 0\;\;[n\to+\infty]\},$$
where $\supp \sigma=\{n\in\N:\sigma(n)\neq n\}$. Note that since $\mu(\{0,...,n\})<+\infty$ for every $n\in\N$, the support of every $\sigma\in\mathfrak S_\lambda$ has finite measure. Also, if $\lambda$ is actually a measure, then $\mathfrak S_\lambda=\{\sigma\in\mathfrak S_\infty: \lambda(\supp \sigma)<+\infty\}$; furthermore if $\lambda$ is a probability measure then $\mathfrak S_\lambda=\mathfrak S_\infty$.

The group $\mathfrak S_\lambda$ is equipped with a natural left-invariant metric $d_\lambda$ analogous to the uniform metric on $\Aut^*(X,\mu)$ defined by 
$$d_\lambda(\sigma,\sigma')=\lambda(\{n\in\N: \sigma(n)\neq\sigma'(n)\}).$$
Note that the condition $\lambda(\supp \sigma\setminus\{0,...,n\})\to 0$ ensures that the countable group of permutations of finite support is dense in $\mathfrak S_\lambda$ which is thus separable.
It is a theorem of Tsankov that $\mathfrak S_\lambda$ is actually a Polish group. The following result is a straightforward adaptation of Theorem \ref{thm:AutQNA}, replacing $d_u$ by $d_\lambda$.

\begin{prop}\label{prop:Slambda QNA}Let $\lambda$ be a lower semi-continuous submeasure on $\N$. Then $\mathfrak S_\lambda$ is quasi non-archimedean. \qed
\end{prop}
 
It is easily checked that the topology of $\mathfrak S_\lambda$ refines the topology induced  by $\mathfrak S_\infty$, so that $\mathfrak S_\lambda$ is always totally disconnected, and that the open subgroups of $\mathfrak S_\lambda$ separate points from the identity (in particular, $\mathfrak S_\lambda$ is not locally generated). The following example shows that it can furthermore fail to be non-archimedean. Note that this is just a particular case of a more general phenomenon: one can actually characterize when the topology fails to be zero-dimensional\footnote{A topology is zero-dimensional if it has a basis made of clopen sets.} (see \cite[Thm. 5.3]{MR2213623}; the example below is taken from \cite[Cor. 4]{Malicki:2015fj}).

\begin{exam}
Consider the measure $\lambda$ on $\N$ defined by 
$$\lambda(A)=\sum_{n\in A} \frac 1n.$$
Then $\mathfrak S_\lambda$ is not a non-archimedean group. Indeed, if we fix $\epsilon>0$ and $N\in\N$, we can find a finite family $(A_i)_{i=1}^N$ of disjoint subsets  of $\N$ such that for every $i\in\{1,...,N\}$, $\frac \epsilon 2<\lambda(A_i)<\epsilon$. For all $i\in\{1,...,N\}$, let $\sigma_i$ be a permutation whose support is equal to $A_i$. Then $\sigma=\prod_{i=1}^N\sigma_i$ is at distance at least $N\epsilon/2$ from the identity, so that the ball of radius $\epsilon$ around the identity generates a group which contains elements arbitrarily far away from the identity. 
\end{exam}

\begin{cor}\label{Corollary:exists td qna not na}There exists a totally disconnected Polish group which is quasi non-archimedean, but not non-archimedean. \qed
\end{cor}

Let us now give examples of totally disconnected Polish group which are infinitesimally finitely generated. To do this, we will upgrade a result of Stevens \cite{zbMATH03966486} who showed the existence of totally disconnected infinitesimally generated Polish groups: we will show that her examples are actually infinitesimally finitely generated. This will be a consequence of the following general statement, which also implies the well-known fact that $\R$ has a dense $G_\delta$ of pairs of topological generators.

\begin{thm}\label{thm: Z 1/2 is lt2g}
Let $G$ be an abelian Polish group which contains the group $\Z[1/2]$ of dyadic rationals as a dense subgroup, and assume furthermore that in the topology induced by $G$ on $\Z[1/2]$, we have $\frac 1{2^n}\to 0$ $[n\to+\infty]$. Then for every $g_0\in\Z[1/2]$, the set of $h\in G$ such that $\la g, h\ra$ generate a dense subgroup of $G$ is dense.

In particular there is a dense $G_\delta$ set of couples of topological generators of $G$ in $G^2$ and so $G$ is infinitesimally finitely generated with infinitesimal rank at most $2$.
\end{thm}
\begin{proof}
Let $g_0\in\Z[1/2]$. In order for a couple $(g_0,h)\in G^2$ to generate a dense subgroup of $G$, it suffices for the closed subgroup they generate to contain $\frac 1{2^n}$ for every $n\in\N$, for the group $\Z[1/2]$ is dense in $G$. So the set $T:=\{h\in G^2:\overline{\la g_0,h\ra}=G\}$ may be written as a countable intersection $T=\bigcap_{n\in\N}T_n$, where
$$T_n:=\left\{h\in G: \frac 1{2^n}\in\overline{\la g_0,h\ra}\right\}.$$
Since $G$ is Polish the set $T_n$ is $G_\delta$, so we only need to show that $T_n$ is dense. To this end, fix $\epsilon>0$, $n\in\N$, and let $h_0\in \Z[1/2]$. We want to find $h\in T_n$ such that $d(h,h_0)<\epsilon$, where $d$ is a fixed compatible metric on $G$. 

Write $g_0=\frac{k_1}{2^{m}}$ and $h_0=\frac{k_2}{2^{m}}$, where $k_1, k_2\in \Z$ and $m\in\N$. We will find $\beta\in\N$ such that for all $N\in\N$, if $h=h_0+\frac \beta{2^{m+N}}$, then $\la g_0,h\ra$ contains $\frac 1{2^{N+m}}$, so that in particular it contains $\frac 1{2^n}$ as soon as $N+m\geq n$. The group $\la g_0,h\ra$  contains $\frac 1{2^{N+m}}$ if and only if we can find $u,v\in\Z$ such that $ug_0+vh=\frac 1{2^{m+N}}$. This condition can be rewritten as
$$2^Nk_1u+(2^{N}k_2+\beta)v=1.$$
So we want to find $\beta\in\N$ such that for all $N\in\N$ we have that $2^Nk_1$ and $2^{N}k_2+\beta$ are relatively prime. Let us furthermore ask that $\beta$ is odd, so that we only have to make sure that every odd prime divisor of $k_1$ does not divide $2^Nk_2+\beta$.

Let $p_1,...,p_k$ list the odd primes which divide both $k_1$ and $k_2$, while $p_{k+1},...,p_l$ are the odd primes which divide $k_1$ but not $k_2$. Then it is easily checked that $\beta=(2+p_{1}\cdots p_k)p_{k+1}\cdots p_l$ works: for all $i\leq k$ we have that $\beta$ is invertible modulo $p_i$ and $p_i$ divides $k_2$ so that $2^Nk_2+\beta$ is not divisible by $p_i$, while for $k<i\leq l$, $\beta$ is null modulo $p_i$ while $2^Nk_2$ is invertible so that $2^Nk_2+\beta$ is not divisible by $p_i$. 

But then, since $\frac 1{2^{m+N}}$ tends to zero as $N$ tends to $+\infty$, we also have $\frac \beta{2^{m+N}}\to 0$ $[N\to+\infty]$. Then, as explained before, the group generated by $g:=g_0$ and $h:=h_0+\frac\beta{2^{m+N}}$ contains $\frac1{2^{m+N}}$, so that $(g,h)\in T_n$, while $0=d(g,g_0)<\epsilon$ and $d(h,h_0)<\epsilon$ if $N$ was chosen large enough. 

So every $T_n$ is a dense subset of $G$, which ends the proof since this furthermore shows that the set of couples  generating a dense subgroup of $G$ is dense in $G^2$ and this set has to be a  $G_\delta$. 
\end{proof}

Let us now apply the previous theorem and describe Steven's examples of Polish groups which are infinitesimally finitely generated but totally disconnected. These groups arise as a Polishable subgroups of the real line, constructed by taking a completion of the dyadic rationals with respect to a well chosen norm which makes $2^{-n}$ have much bigger norm than usually.

We fix a biinfinite sequence of positive real number $(r_i)_{i\in\Z}$ such that $r_i\to 0 [i\to+\infty]$ and for all $i\in\Z$, we have $r_{i+1}\leq r_i\leq 2r_{i+1}$. Then one can define the following group norm\footnote{A norm on an abelian group is a function $\abs\cdot: G\to [0,+\infty)$ such that for any $x,y\in G$, $\abs{x+y}\leq \abs x+\abs y$, and $\abs x=\abs{-x}$.} 
$\norm\cdot$ on the ring $\Z[1/2]$ of dyadic rationals: for every $x\in\Z[1/2]$,
$$\norm x:= \inf\left\{\sum_{i=-n}^n \abs{a_i} r_i: x=\sum_{i=-n}^n a_i2^{-i}, a_i\in\Z, n\in\N\right\}.$$
It is easy to check that this defines a group norm on $\Z[1/2]$ which refines the usual norm. Using the fact that $r_i\leq 2r_{i+1}$, one can easily show that for all $x\in\Z[1/2]$,
$$
 \norm x= \inf\left\{\sum_{i=-n}^n \abs{a_i} r_i: x=\sum_{i=-n}^n a_i2^{-i}, a_i\in\{-1,0,1\}, n\in\N\right\}.
$$
In particular, we see that for all $n\in\N$, we have $\norm{2^{-n}}=r_n$ so that $2^{-n}\to 0$ as $n\to+\infty$. Let $\overline{\Z[1/2]}^{\norm\cdot}$ denote the completion of $\Z[1/2]$ with respect to this norm. Since this norm refines the usual norm, $\overline{\Z[1/2]}^{\norm\cdot}$ is a subgroup of $\R$. Stevens explicitely described the elements of $\R$ belonging to $\overline{\Z[1/2]}^{\norm\cdot}$ and showed that the group $\overline{\Z[1/2]}^{\norm\cdot}$ is infinitesimally generated \cite[Thm. 2.1 (ii)]{zbMATH03966486}, and we see that Theorem \ref{thm: Z 1/2 is lt2g} strengthens this because it implies that $\overline{\Z[1/2]}^{\norm\cdot}$ is infinitesimally $2$-generated.

To obtain totally disconnected examples, we need another result of Stevens stating that the following are equivalent (see \cite[Thm. 2.2]{zbMATH03966486}):
\begin{enumerate}[(i)]
\item $\sum_{i\in\N} r_i=+\infty$,
\item $\norm\cdot$ is not equivalent to $\abs\cdot$ when restricted to $\Z[1/2]$,
\item $\Q\cap \overline{\Z[1/2]}^{\norm\cdot}=\Z[1/2]$,
\item $\overline{\Z[1/2]}^{\norm\cdot}$ is totally disconnected,
\end{enumerate}

\begin{rmq}
Note that every subgroup $G$ of $\R$ which is not equal to $\R$ has to be totally disconnected for the induced topology, since its complement is dense in $\R$ so that the sets of the form $]r,+\infty[$ for $r\in \R\setminus G$ are clopen in $G$. In particular, if $G$ is a proper subgroup of $\R$ equipped with a topology which refines the usual topology of $\R$ then $G$ is totally disconnected. 
So conditions (ii) and condition (iii) clearly imply condition (iv).
\end{rmq}

So suppose further that $\sum_{i\in\N} r_i=+\infty$  (e.g. take $r_i=1$ if $i\leq 0$ and $r_i=\frac 1i$ otherwise). Then we see that $\overline{\Z[1/2]}^{\norm\cdot}$ is a totally disconnected Polish group which has infinitesimal rank at most $2$. Moreover since this group is a subgroup of the real line endowed with a finer topology (see \cite[Thm. 2.1]{zbMATH03966486}) it cannot be monothetic, so its infinitesimal rank is actually equal to $2$.

\begin{cor}\label{cor: there is td not QNA}There exists a totally disconnected Polish group which has infinitesimal rank $2$, in particular there is a totally disconnected Polish group which is not quasi non-archimedean. 
\end{cor}

\begin{rmq}
Note that one can see directly that Stevens' groups are not quasi non-archimedean, even for $n=1$. Indeed, if $U$ is a neighborhood of $0$ not containing $1$, then if $V$ is another neighborhood of $0$ there is some $N\in\N$ such that $1/2^N\in V$, but the group generated by $1/2^N$ contains $1$ hence it is not a subset of $U$.
\end{rmq}

We know that while the group of the reals has infinitesimal rank $2$, its quotient $\mathbb S^1=\R/\Z$ has infinitesimal rank $1$. The same is true of Stevens' examples, which is going to yield the following result.

\begin{thm}\label{thm:infrank1 but td}There exists a totally disconnected Polish group which has infinitesimal rank $1$.
\end{thm}
\begin{proof}
Let $G$ be a totally disconnected Polish group obtained by Stevens' construction from a sequence ; then $G$ is a proper subgroup of $\mathbb R$ containing $\Z[1/2]$ as a dense subgroup. Observe that $\Z$ is a discrete subgroup of $G$ and we may thus form the Polish group $\tilde G:=G/\Z$.

The group $\tilde G$ is a proper dense subgroup of $\mathbb S^1=\R/\Z$, so $\mathbb S^1\setminus \tilde G$ is thus dense in $\mathbb S^1$. Let $A=p([0,1/2[)$ where $p: \R\to \R/\Z$ is the usual projection, then for all $g\in\mathbb S^1\setminus \tilde G$ the set $(g+A)\cap \tilde G$ is clopen in $\tilde G$. Moreover since $\mathbb S^1\setminus \tilde G$ is dense in $\mathbb S^1$ the family of sets $((g+A)\cap \tilde G)_{g\in \mathbb S^1\setminus \tilde G}$ separates points in $\tilde G$, so $\tilde G$ is totally disconnected. 

Furthermore, we have by Theorem \ref{thm: Z 1/2 is lt2g} that there is a dense $G_\delta$ of $h\in G$ such that the group generated by $1$ and $h$ is dense in $G$. Since $1\in\Z$ and $\tilde G=G/\Z$ we conclude that there is a dense $G_\delta$ of $h\in\tilde G$ which generate a dense subgroup in $\tilde G$, in particular $\tilde G$ has infinitesimal rank $1$.
\end{proof}

We don't know an example of a totally disconnected Polish group which is infinitesimally generated and quasi non-archimedean. Moreover, we want to stress out that all the examples we know of Polish groups which are quasi non-archimedean actually fail the property even for $n=1$, so it would be very interesting to have examples having a ‘‘non-QNA rank'' greater than $1$.

\section{Further remarks and questions}\label{sec:remarks}

Let us point out how one can easily  build Polish groups into which no non-discrete locally compact group can embed.

\begin{lem}Let $\Gamma$ be a countable discrete group without elements of finite order. Then every monothetic subgroup of $\LL^0(X,\mu,\Gamma)$ is infinite discrete. In particular, no nontrivial compact group embeds into $\LL^0(X,\mu,\Gamma)$.
\end{lem}
\begin{proof}
Given $1\neq f\in\LL^0(X,\mu,\Gamma)$, find $A\subseteq X$ non-null and $\gamma\in\Gamma\setminus\{1\}$ such that $f_{\restriction A}$ is constant equal to $\gamma$. By assumption, for all $n\in\Z\setminus \{0\}$, the support of $f^n$ contains $A$, and so $\la f\ra$ is discrete. 
\end{proof}

\begin{thm} Let $\Gamma$ be a countable discrete group without elements of finite order, and let $G$ be a separable locally compact group. Then every continuous morphism $G\to \LL^0(X,\mu,\Gamma)$ factors through a discrete group.
\end{thm}
\begin{proof}
Let $G^0$ be the connected component of the identity. Because $\LL^0(X,\mu,\Gamma)$ is quasi non-archimedean, $\pi$ factors through $G/G^0$ by Proposition \ref{prop:QNAtrivial} and Theorem \ref{thm:infinitesimally finitely generated}. Then by van Dantzig's theorem and the previous lemma, the kernel of the later map contains an  open subgroup of $G/G^0$, hence it factors through a discrete group.
\end{proof}

\begin{qu}Is there a Polish group without any locally compact closed subgroup?
\end{qu}

The group $\LL^0(X,\mu,G)$ was originally introduced to show that every Polish group embeds into a connected group, and we saw that being quasi non-archimedean is somehow opposite to being connected. Because $\LL^0(X,\mu,G)$ can be quasi non-archimedean, one may ask whether every Polish group embeds into a connected not quasi non-archimedean group. The isometry group of the Urysohn space answers this question --- it is universal for Polish groups, connected (see \cite{zbMATH05012574,zbMATH05691892} for stronger versions of these results as well as background on the Urysohn space) and cannot be quasi non-archimedean by universality. 

However, the same question can be asked replacing not quasi non-archimedean by infinitesimally finitely generated. It seems to be open wether the isometry group of the Urysohn space is infinitesimally finitely generated (it is topologically $2$-generated by a result of Solecki, see \cite{MR2255813}).

\begin{qu} Is the isometry group of the Urysohn space infinitesimally $2$-generated?
\end{qu}

Let us end this paper by mentioning a question related to ample generics. A Polish group $G$ has \textbf{ample generics} if the diagonal conjugacy action of $G$ onto $G^n$ has a comeager orbit for every $n\in\N$ (see \cite{MR2308230}). It has been recently discovered that there exists Polish groups with ample generics which are not non-archimedean (see\cite{MR3431579}  and \cite{Malicki:2015fj}). These examples arise either as full groups or as groups of the form $\mathfrak S_\lambda$, which are quasi non-archimedean groups by Theorem \ref{thm:AutQNA} and Proposition \ref{prop:Slambda QNA}. This motivates the following question. 

\begin{qu}Is there a Polish group which has ample generics, but which is not quasi non-archimedean?
\end{qu}


\bibliographystyle{alpha}
\bibliography{Feb2016}

\begin{thebibliography}{{Kap}71}

\bibitem[Aue34]{Auerbach1934}
H.~Auerbach.
\newblock Sur les groupes lin{\'e}aires born{\'e}s (iii).
\newblock {\em Studia Mathematica}, 5(1):43--49, 1934.

\bibitem[BG03]{zbMATH01903603}
E.~{Breuillard} and T.~{Gelander}.
\newblock {On dense free subgroups of Lie groups.}
\newblock {\em {J. Algebra}}, 261(2):448--467, 2003.

\bibitem[BGSS06]{MR2255501}
Emmanuel Breuillard, Tsachik Gelander, Juan Souto, and Peter Storm.
\newblock Dense embeddings of surface groups.
\newblock {\em Geom. Topol.}, 10:1373--1389, 2006.

\bibitem[BK96]{MR1425877}
Howard Becker and Alexander~S. Kechris.
\newblock {\em The descriptive set theory of {P}olish group actions}, volume
  232 of {\em London Mathematical Society Lecture Note Series}.
\newblock Cambridge University Press, Cambridge, 1996.

\bibitem[CI77]{MR0447455}
W.~W. Comfort and G.~L. Itzkowitz.
\newblock Density character in topological groups.
\newblock {\em Math. Ann.}, 226(3):223--227, 1977.

\bibitem[Gao09]{MR2455198}
Su~Gao.
\newblock {\em Invariant descriptive set theory}, volume 293 of {\em Pure and
  Applied Mathematics (Boca Raton)}.
\newblock CRC Press, Boca Raton, FL, 2009.

\bibitem[{Gel}08]{zbMATH05508698}
Tsachik {Gelander}.
\newblock {On deformations of $F_n$ in compact Lie groups.}
\newblock {\em {Isr. J. Math.}}, 167:15--26, 2008.

\bibitem[G{\.Zuk}02]{zbMATH01760707}
Tsachik {Gelander} and Andrzej {\.Zuk}.
\newblock {Dependence of Kazhdan constants on generating subsets.}
\newblock {\em {Isr. J. Math.}}, 129:93--98, 2002.

\bibitem[HM90]{MR1082789}
Karl~Heinrich Hofmann and Sidney~A. Morris.
\newblock Weight and {$c$}.
\newblock {\em J. Pure Appl. Algebra}, 68(1-2):181--194, 1990.
\newblock Special issue in honor of B. Banaschewski.

\bibitem[HS42]{MR0006543}
Paul~R. Halmos and H.~Samelson.
\newblock On monothetic groups.
\newblock {\em Proc. Nat. Acad. Sci. U. S. A.}, 28:254--258, 1942.

\bibitem[{Kap}71]{zbMATH03353499}
Irving {Kaplansky}.
\newblock {\em {Lie algebras and locally compact groups.}}
\newblock Chicago Lectures in Mathematics. 1971.

\bibitem[Kec95]{MR1321597}
Alexander~S. Kechris.
\newblock {\em Classical descriptive set theory}, volume 156 of {\em Graduate
  Texts in Mathematics}.
\newblock Springer-Verlag, New York, 1995.

\bibitem[Kec10]{MR2583950}
Alexander~S. Kechris.
\newblock {\em Global aspects of ergodic group actions}, volume 160 of {\em
  Mathematical Surveys and Monographs}.
\newblock American Mathematical Society, Providence, RI, 2010.

\bibitem[KLM15]{MR3431579}
Adriane Ka{\"\i}chouh and Fran{\c c}ois Le~Ma{\^\i}tre.
\newblock Connected {P}olish groups with ample generics.
\newblock {\em Bull. Lond. Math. Soc.}, 47(6):996--1009, 2015.

\bibitem[KR07]{MR2308230}
Alexander~S. Kechris and Christian Rosendal.
\newblock Turbulence, amalgamation, and generic automorphisms of homogeneous
  structures.
\newblock {\em Proc. Lond. Math. Soc. (3)}, 94(2):302--350, 2007.

\bibitem[{Kro}31]{zbMATH03002825}
Leopold {Kronecker}.
\newblock {\em N{\"a}herungsweise ganzzahlige Aufl{\"o}sung linearer
  Gleichungen}.
\newblock Leopold Kronecker's Werke. Hrsg. von K. Hensel. Bd. 3. 1. Halbbd.
  Leipzig : Druck und Verlag von B.G. Teubner, 1931.

\bibitem[{Kur}49]{zbMATH03069311}
Masatake {Kuranishi}.
\newblock {Two elements generations on semi-simple Lie groups.}
\newblock {\em {K\=odai Math. Semin. Rep.}}, 1949:89--90, 1949.

\bibitem[LM14]{gentopergo}
Fran{\c c}ois Le~Ma{\^\i}tre.
\newblock The number of topological generators for full groups of ergodic
  equivalence relations.
\newblock {\em Invent. Math.}, 198:261--268, 2014.

\bibitem[Mal15]{Malicki:2015fj}
Maciej Malicki.
\newblock An example of a non non-archimedean polish group with ample generics.
\newblock {\em to appear in Proc. Amer. Math. Soc.}, 2015.

\bibitem[Mau81]{MR666400}
R.~Daniel Mauldin, editor.
\newblock {\em The {S}cottish {B}ook}.
\newblock Birkh{\"a}user, Boston, Mass., 1981.
\newblock Mathematics from the Scottish Caf{{\'e}}, Including selected papers
  presented at the Scottish Book Conference held at North Texas State
  University, Denton, Tex., May 1979.

\bibitem[{Mel}06]{zbMATH05012574}
Julien {Melleray}.
\newblock {Stabilizers of closed sets in the Urysohn space.}
\newblock {\em {Fundam. Math.}}, 189(1):53--60, 2006.

\bibitem[{Mel}10]{zbMATH05691892}
Julien {Melleray}.
\newblock {Topology of the isometry group of the Urysohn space.}
\newblock {\em {Fundam. Math.}}, 207(3):273--287, 2010.

\bibitem[Sol99]{MR1708146}
S{\l}awomir Solecki.
\newblock Analytic ideals and their applications.
\newblock {\em Ann. Pure Appl. Logic}, 99(1-3):51--72, 1999.

\bibitem[Sol05]{MR2255813}
S{\l}awomir Solecki.
\newblock Extending partial isometries.
\newblock {\em Israel J. Math.}, 150:315--331, 2005.

\bibitem[{Ste}86]{zbMATH03966486}
T.Christine {Stevens}.
\newblock {Connectedness of complete metric groups.}
\newblock {\em {Colloq. Math.}}, 50:233--240, 1986.

\bibitem[SU35]{Schreier1935}
J.~Schreier and Stanis{\l}aw Ulam.
\newblock Sur le nombre des g{\'e}n{\'e}rateurs d'un groupe topologique compact
  et connexe.
\newblock {\em Fundamenta Mathematicae}, 24(1):302--304, 1935.

\bibitem[Tsa06]{MR2213623}
Todor Tsankov.
\newblock Compactifications of {$\Bbb N$} and {P}olishable subgroups of
  {$S_\infty$}.
\newblock {\em Fund. Math.}, 189(3):269--284, 2006.

\end{thebibliography}

\end{document}